\documentclass{amsart}

\usepackage{amssymb}
\usepackage{amsmath}
\usepackage{amsthm}
\usepackage{amsbsy}
\usepackage{bm}
\usepackage[dvips]{graphicx}

\theoremstyle{plain}
\newtheorem{theorem}{Theorem}[section]
\newtheorem{lemma}[theorem]{Lemma}

\theoremstyle{definition}
\newtheorem{definition}[theorem]{Definition}

\theoremstyle{remark}

\newtheorem{claim}{Claim}

\newtheorem*{acknowledgements}{Acknowledgements}

\newcommand{\Z}{\mathbb{Z}}

\newcommand{\del}{\partial}

\begin{document}

\title{Incompressibility and normal minimal surfaces}

\author{Tejas Kalelkar}

\address{   Stat-Math Unit\\
        Indian Statistical Institute\\
        Bangalore, India}

\email{tejas@isibang.ac.in}

\date{\today}


\maketitle
\begin{abstract}
In this paper we describe a procedure for refining the given triangulation of a 3-manifold that scales the PL-metric according to a given weight function while creating no new normal surfaces.

It is known that an incompressible surface $F$ in a triangulated 3-manifold $M$ is isotopic to a normal surface that is of minimal PL-area in the isotopy class of $F$. Using the above scaling refinement we prove the converse. If $F$ is a surface in a closed 3-manifold $M$ such that for any triangulation $\tau$ of $M$, $F$ is isotopic to a $\tau$-normal surface $F(\tau)$ that is of minimal PL-area in its isotopy class, then we show that $F$ is incompressible.
\end{abstract}

\section{Introduction}
Given a Riemannian manifold $(M, g)$, we can scale the metric by multiplying $g$ with a smooth positive real-valued function. Such a rescaling may, however, introduce new minimal surfaces. Given a triangulated 3-manifold $(M, \tau)$, we can scale the PL-metric by taking a refinement of $\tau$, by repeatedly subdividing the tetrahedra in $\tau$ according to a positive integer-valued scaling function. In general, such a scaling may introduce new minimal normal surfaces. We describe here a procedure for scaling the PL-metric that introduces no new normal surfaces.

\begin{definition}\label{refinement}
Let $\Delta$ be a tetrahedron with vertices labeled $\{ a, b, c, d \}$. Let $e$ be a point in the
interior of $\Delta$. Take a simplicial triangulation of $\Delta$ using the
tetrahedra $\Delta_A = [b, c, d, e], \Delta_B = [a, c, d, e], \Delta_C = [a, b, d, e]$ and $\Delta_D = [a, b, c, e]$. Define $\phi$ on a triangulation $\tau$ to be the function that gives a refinement of $\tau$ by dividing each tetrahedron $\Delta$ of $\tau$ into 4 tetrahedra, as described above. We call this the refinement function. This is shown in Figure \ref{Fig1} where the additional edges in the refinement of $\Delta$ are shown as dotted-lines.

Let $f: \{\Delta : \Delta \in \tau \} \rightarrow \Z$ be a function that associates a non-negative integer to each tetrahedron $\Delta$ of $\tau$. We call such a function a scaling function. Define $\phi_f$ on $\tau$ to be the function that gives to each $\Delta \in \tau$ the triangulation $\phi^{f(\Delta)}(\Delta)$, obtained by taking $f(\Delta)$ iterates of $\phi$ on $\Delta$. As the faces of $\Delta$ are also faces of $\phi(\Delta)$, $\phi_f(\tau) = \tau'$ is a refined triangulation of $\tau$.
\end{definition}

\begin{figure}
\centering
\includegraphics[width=0.3\textwidth]{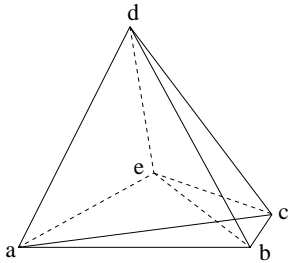}
\caption{A tetrahedron in $\tau$, partitioned by $\phi$.}\label{Fig1}
\end{figure}
The main theorem in this paper is Theorem \ref{mainthm}.

\begin{theorem}\label{mainthm}
Let $F$ be a closed surface embedded in a 3-manifold $M$ no component of which
is a 2-sphere. Let $\tau$ be a triangulation of $M$. Let $f: \{\Delta : \Delta \in \tau \} \rightarrow \Z$ be a scaling function and let $\tau' = \phi_f(\tau)$ be the corresponding refinement of $\tau$. Then, $F$ is $\tau$-normal $\Leftrightarrow$ $F$ is
$\tau'$-normal.
\end{theorem}

Every $\tau$-normal surface is $\tau'$-normal by observing that each $\tau$-normal disk is a union of $\tau'$-normal disks, as shown in lemma \ref{disks}. For the converse, the proof depends on a simple examination of the possible $\tau'$-normal proper embeddings of a surface in a tetrahedron $\Delta$ of the triangulation $\tau$. We show that every $\tau'$-normal surface within $\Delta$ is in fact a $\tau$-normal disk, hence every $\tau'$-normal surface is also $\tau$-normal. \\

In the second part of this paper we use such a refinement of the triangulation to obtain a PL-analogue of the theorem proved in \cite{Ga}. It is known that if $F$ is a smooth incompressible surface in an irreducible Riemannian 3-manifold $M$, then the isotopy class of $F$ has a least area surface. The theorem proved by Gadgil in \cite{Ga} proves the converse, that is, if $F$ is a smooth surface in a closed, irreducible 3-manifold $M$ such that for each Riemannian metric $g$ of $M$, $F$ is isotopic to a least-area surface $F(g)$, then $F$ is incompressible.

Similarly, in the PL case, it is known that an incompressible surface $F$ in a triangulated 3-manifold $M$ is isotopic to a normal surface that is of minimal PL-area in the isotopy class of $F$. We prove here the converse.
\begin{theorem}\label{incomprthm}
Let $F$ be a closed orientable surface in an irreducible orientable closed 3-manifold $M$. Then, $F$ is incompressible
if and only if for any triangulation $\tau$ of $M$, there exists
a $\tau$-normal surface $F(\tau)$ isotopic to $F$ that is of minimal PL-area in the
isotopy class of $F$.
\end{theorem}

If $F$ is an incompressible surface that is not normal in a triangulation $\tau$ of $M$, then it is known that a PL-area decreasing isotopy exists. To prove the converse, we show that given a compressible surface $F$, there exists a triangulation $\tau'$ for which the isotopy class of $F$ has no normal minimal surface. An outline of the proof is as follows. Let $\hat{F}$ be the surface obtained by compressing $F$ along a compressing disc. Therefore, $F$ is obtained from $\hat{F}$ by attaching a 1-handle to $\hat{F}$. We start with a certain `prism' triangulation of a regular neighbourhood $N(\hat{F})$ of $\hat{F}$ which is such that any connected normal surface lying in $N(\hat{F})$ is isotopic to a component of $\hat{F}$. We extend this triangulation to a triangulation $\tau$ of $M$. Let $Ar(F)$ be the PL-Area of $F$ in $\tau$, then as the 1-handle can be chosen to avoid all edges of $\tau$, we always get a representative of $F$ (in its isotopy class) such that $Ar(F) = Ar(\hat{F})$.

We define a scaling function $f: \{ \Delta : \Delta \in \tau \} \rightarrow \Z$ that takes the value 0 on $\Delta \subset N(\hat{F})$ and a value greater than $Ar(\hat{F})$ for $\Delta$ not in $N(\hat{F})$. We now take the refinement $\tau'$ of $\tau$ given by $\phi_f$. As every $\tau'$-normal surface is also $\tau$-normal by Theorem \ref{mainthm}, a $\tau'$-normal surface that does not lie in $N(\hat{F})$ has a $\tau$-normal disk outside $N(\hat{F})$. This disk has $\tau'$ PL-area more than the $\tau'$ PL-area of $\hat{F}$. As $\hat{F}$ is not homeomorphic to $F$, a normal surface that lies entirely in $N(\hat{F})$ is not isotopic to $F$, while a normal surface that does not lie in $N(\hat{F})$ has $\tau'$ PL-area more than that of $\hat{F}$. As there is always a surface isotopic to $F$ that has $\tau'$ PL-area equal to that of $\hat{F}$ (and which is not normal) so, the isotopy class of $F$ has no normal minimal-area surfaces in $\tau'$.

\section{Proof of Theorem \ref{mainthm}}
In this section, we first show lemma \ref{sets} which determines normal disks using normal arcs in the boundary of the disk. Then, using Fig \ref{Fig2} we prove lemma \ref{disks} which says that a $\tau$-normal disk is a $\tau'$-normal surface. We then prove Theorem \ref{mainthm}. We introduce the following notation:
\begin{definition}
A $\tau'$-normal triangle $T$ in $\Delta_X \subset \Delta, X \in \{ A, B, C, D \}$, is said to link a vertex
$w$ in $\Delta_X$ if $\del \Delta_X - \del T$ has a component that contains the vertex $w$
and no other vertices of $\Delta_X$. We say the coordinates of $T$ are $[T] = (X, T_w)$. When the context is clear we shall denote the triangle $T$ itself by its coordinates $(X, T_w)$.

Similarly, a $\tau'$-normal quadrilateral $Q$ in $\Delta_X \subset \Delta, X \in \{ A, B, C, D \}$, is said to link an
edge $yz$ in $\Delta_X$ if $\del \Delta_X - \del Q$ has a component that contains the
vertices $y$ and $z$, and no other vertices of $\Delta_X$. We say the coordinates of $Q$ are $[Q] = (X, Q_{yz})$. When the context is clear we shall denote the quadrilateral $Q$ itself by its coordinates $(X, Q_{yz})$.
\end{definition}

\begin{definition}
A $\tau'$-normal arc $\lambda$ is said to link a vertex $x$
(respectively an edge $yz$) in a face $F$ of $\tau'$ if $F - \lambda$ has a
component that contains the vertex $x$ and no other vertices of
$F$ (respectively contains the vertices $y$ and $z$ and no other
vertices of $F$). Denote the set of $\tau'$-normal arcs in faces of tetrahedra of $\tau'$,
linking vertex $x$ (respectively edge $yz$) by $\Lambda_x$ (respectively $\Lambda_{yz}$).
\end{definition}

\begin{definition}
We define $\Lambda_x * \Lambda_y$ (respectively $\Lambda_x *
\Lambda_{xy})$ to be the set of $\tau'$-normal paths that are not
contained in a single face, and are given by the concatenation of
an arc in $\Lambda_x$ with an arc in $\Lambda_y$ (respectively
$\Lambda_{xy}$).
\end{definition}

We now state the following lemma which says that given a pair of contiguous normal arcs in the boundary of a normal disc, we can determine whether the disc is a triangle or a quadrilateral and we can determine which vertex (respectively which edge) it links. Also, if we are given that the normal disk is a triangle (respectively a quadrilateral) and we are given one normal arc in its boundary, then the vertex linked by the normal triangle (respectively the edge linked by the normal quadrilateral) can be determined.

\begin{lemma} \label{sets}
For a $\tau'$-normal disk $D$ with a normal path $\lambda
\subset \del D$,\\
(i) If $D$ is a triangle with $\lambda \in \Lambda_x$ then $D$ is
a triangle linking the vertex $x$.\\
(ii) If $D$ is a quadrilateral with $\lambda \in \Lambda_{xy}$
then $D$ is a quadrilateral linking the edge $xy$.\\
(iii) If $D$ is a quadrilateral in the tetrahedron $[w, x, y,
z]$ with $\lambda$ in the face $[x, y, z]$ and $\lambda \in
\Lambda_z$, then $D$ is a quadrilateral linking the edge $wz$.\\
(iv) If $\lambda \in \Lambda_x * \Lambda_x$ then $D$ is a
triangle linking the vertex $x$. \\
(v) If $\lambda \in \Lambda_x * \Lambda_{xy}$ then $D$ is a
quadrilateral linking the edge $xy$.
\end{lemma}

\begin{proof}
We make the following simple observations:\\
a. A normal disk $D$ is a quadrilateral if and only if each normal
arc in $\del D$ links a distinct vertex, which is the same as saying normal
arcs in $\del D$ link more than one vertex.\\
b. A normal triangle $T$ in a tetrahedron $[ w, x, y, z ]$ links
the vertex $x$ if and only if any normal arc in $\del T$ links $x$.\\
c. A normal quadrilateral $Q$ in a tetrahedron $[w, x, y, z ]$
links edge $xy$ if and only if $\del Q \cap xy = \phi$.\\

The statement (i) follows from observation b.

Let $D$ be a quadrilateral in a tetrahedron $[ w, x, y, z ]$,
with $\lambda \in \Lambda_{xy}$ and $\lambda$ contained in the
face $F = [ x, y, z ]$. Then $\del
\lambda$ is contained in the arcs $xz$ and $yz$. Therefore, if
$\del D \cap xy \neq \phi$ then as $\del D$ is transverse to
edges, $\del D$ is a circle in $\del \Delta$ transversely intersecting each edge of the triangle $[x, y, z]$. Therefore, $\del D$ must intersect some edge of this triangle more than once. This is a contradiction as $D$ is a normal disc, so that $\del D$ intersects each edge at most once.
Therefore, $\del D \cap xy = \phi$ and statement (ii) follows from observation c.

Statement (iii) follows from a similar argument replacing
$\Lambda_{xy}$ with $\Lambda_z$ and observing that quadrilaterals
that link $xy$ are precisely the quadrilaterals that link $wz$.

Statement (iv) follows from observations a and b.

The disk $D$ in statement (v) is a quadrilateral from observation a.
As there exists an arc $\lambda \subset \del D$ with $\lambda \in
\Lambda_{xy}$, from statement (ii) we can see that $D$ links the
edge $xy$.
\end{proof}

\begin{figure}
\centering
\includegraphics[width=0.7\textwidth]{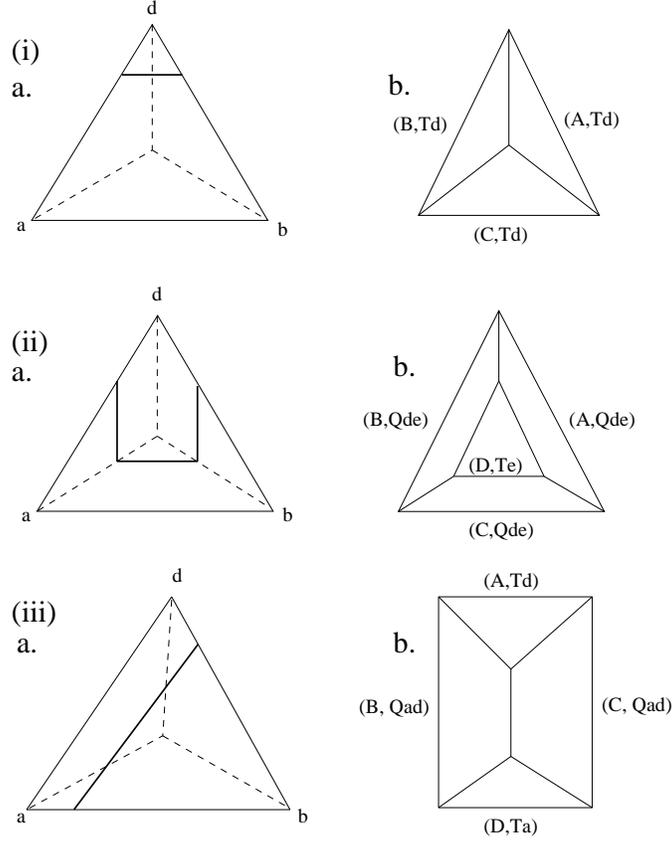}
\caption{Diagrams (i) and (ii) represent a normal triangle linking vertex d. Diagram (iii) represents a normal quadrilateral linking edge ad.}\label{Fig2}
\end{figure}

We now state the lemma which shows that every $\tau$-normal disk is a union of $\tau'$-normal disks.
\begin{lemma}\label{disks}
Let $S$ be a properly embedded surface in $\Delta$. Then,\\
(i) $S$ is a $\tau'$-normal surface with $S = (A, T_d) \cup (B,
T_d) \cup (C, T_d)$ or $S = (A, Q_{de}) \cup (B, Q_{de}) \cup (C,
Q_{de}) \cup (D, T_e)$ $\Leftrightarrow$ $S$ is $\tau$-isotopic to a $\tau$-normal
triangle linking vertex $d$. \\
(ii) $S$ is a $\tau'$-normal surface with $S = (D, T_a) \cup (B,
Q_{ad}) \cup (C, Q_{ad}) \cup (A, T_d)$ or $S = (B, T_c) \cup (D, Q_{bc}) \cup (A, Q_{bc}) \cup (C, T_b)$ $\Leftrightarrow$ $S$ is $\tau$-isotopic to a $\tau$-normal quadrilateral linking edge $ad$.\\
(iii) $S$ is a $\tau'$-normal surface with $S = (A, T_e) \cup (B,
T_e) \cup(C, T_e) \cup(D, T_e)$ $\Leftrightarrow$ $S$ is a $\tau'$ vertex-linking sphere linking the vertex $e$.
\end{lemma}
\begin{proof}
If $S$ is a $\tau$-normal triangle in the tetrahedron $\Delta$, linking vertex $d$, then after a $\tau$-normal isotopy we may assume that $S = \del B(d) \cap \Delta$, where $B(d)$ is a small ball neighbourhood of $d$ in $M$. This is shown in Figure \ref{Fig2} (i) a. Then, $S$ intersects the faces of $\tau'$ as shown in Figure \ref{Fig2} (i) b, so that $S = (A, T_d) \cup (B, T_d) \cup (C, T_d)$. Conversely, if $S = (A, T_d) \cup (B, T_d) \cup (C, T_d)$ (Figure \ref{Fig2} (i) b) or $S = (A, Q_{de}) \cup (B, Q_{de}) \cup (C, Q_{de}) \cup (D, T_e)$ (Figure \ref{Fig2} (ii) b.), then $S$ is a properly embedded disk in $\Delta$ with $\del S$ a circle in $\del \Delta$ linking vertex $d$. So, $S$ is a $\tau$-normal triangle linking vertex $d$.\\

If $S$ is a $\tau$-normal quadrilateral in the tetrahedron $\Delta$, linking edge $ad$, then after a $\tau$-normal isotopy we may assume that $S = \del B(ad) \cap \Delta$, where $B(ad)$ is a small ball neighbourhood of $ad$ in $M$. This is shown in Figure \ref{Fig2} (iii) a. Then, $S$ intersects the faces of $\tau'$ as shown in Figure \ref{Fig2} (iii) b, so that $S = (D, T_a) \cup (B,
Q_{ad}) \cup (C, Q_{ad}) \cup (A, T_d)$. Conversely if $S = (D, T_a) \cup (B, Q_{ad}) \cup (C, Q_{ad}) \cup (A, T_d)$ (Figure \ref{Fig2} (iii) b) or $S = (B, T_c) \cup (D, Q_{bc}) \cup (A, Q_{bc}) \cup (C, T_b)$(corresponding to a quadrilateral linking the edge $bc$), then $S$ is a properly embedded disk in $\Delta$ with $\del S$ a circle in $\del \Delta$ linking the edge $ad$. So, $S$ is a $\tau$-normal quadrilateral linking edge $ad$.\\

If $S$ is a vertex linking sphere linking vertex $e$, then $S = \del B(e)$ where $B(e)$ is a small ball-neighbourhood of $e$ in $M$. So that $S = (A, T_e) \cup (B, T_e) \cup(C, T_e) \cup(D, T_e)$. Conversely, if $S =  (A, T_e) \cup (B, T_e) \cup(C, T_e) \cup(D, T_e)$ then it is easy to see that $S = \del B(e)$ and therefore $S$ is a vertex-linking sphere in $\tau'$ linking vertex $e$.
\end{proof}

We now give a proof of Theorem \ref{mainthm}.
\begin{proof}
If $F$ is $\tau$-normal then by lemma \ref{disks} it follows that $F$ is a union of $\tau'$-normal disks and hence is $\tau'$-normal as well.

To prove the converse, let $S$ be a connected component of $F\cap \Delta$. We shall show in Claims 1 and 2 that if $S$ contains a $\tau'$-normal triangle, then $S$ must either be a $\tau$-normal disk or a vertex-linking sphere. In Claim 3 we shall show that $S$ is not the union of $\tau'$-normal quadrilaterals. So every component of $F \cap \Delta$ is either a $\tau$-normal disk or a $\tau'$ vertex-linking
sphere. Thus we would have shown that any $\tau'$-normal surface in $M$ is either a $\tau$-normal surface
or it has a component which is a $\tau'$ vertex-linking sphere.\\

\begin{claim}
If $S' \subset S$ is a $\tau'$-normal triangle with coordinates $(X, T_e)$ for some $X
\in \{A, B, C, D \}$ then $S$ is either a vertex linking sphere
in $\tau'$ linking vertex $e$ or $S$ is a $\tau$-normal
triangle.\\

Without loss of generality, assume $X=D$. The boundary $\del S'$ is composed of normal arcs linking vertex $e$, i.e., for $Y \in \{A, B,
C \}$, $S' \cap \Delta_Y$ gives a normal arc in $\Lambda_e$.

If $S' \cap \Delta_A$ meets a $\tau'$-normal triangle $T$ in
$\Delta_A$, then by lemma \ref{sets} (i), $[T] = (A, T_e)$. Let
$S'' = S' \cup T$. Now for $Y \in \{ B, C \}$; $S'' \cap \Delta_Y \in
\Lambda_e * \Lambda_e$, therefore by lemma \ref{sets} (iv) $S''$
meets normal triangles with coordinates $(B, T_e)$ and $(C, T_e)$. So as $S$ is connected, $S =
(A, T_e) \cup (B, T_e) \cup (C, T_e) \cup (D, T_e)$, therefore by
lemma \ref{disks}, $S$ is a vertex-linking sphere linking vertex
$e$.

If $S' \cap \Delta_A$ meets a $\tau'$-normal quadrilateral $Q$ in
$\Delta_A$ , then by lemma \ref{sets} (iii), $[Q] = (A, Q_{de})$.
Let $S'' = S' \cup Q$. Now for $Y \in \{ B, C \}$; $S'' \cap \Delta_Y \in
\Lambda_e * \Lambda_{de}$, therefore by lemma \ref{sets} (v),
$S''$ meets normal quadrilaterals with coordinates $(B, Q_{de})$ and $(C,
Q_{de})$. So as $S$ is connected, $S = (A, Q_{de}) \cup (B, Q_{de})\cup (C,
Q_{de}) \cup (D, T_e)$, therefore by lemma \ref{disks}, $S$ is a
$\tau$-normal triangle (linking vertex $d$).
\end{claim}

\begin{claim}
If $S' \subset S$ is a $\tau'$-normal triangle with coordinates $(X, T_w)$, for $X \in
\{A, B, C, D\}$, where $w$ is a vertex in $\Delta_X$ other than
$e$, then $S$ is either a
$\tau$-normal triangle or a $\tau$-normal quadrilateral.\\

Without loss of generality, assume $X=D$ and $w=a$. Then $S' \cap
\Delta_B$ and $S' \cap \Delta_C$ are in $\Lambda_a$, while $S'
\cap \Delta_A = \phi$.

If $S' \cap \Delta_B$ meets a $\tau'$-normal triangle $T$, then by
lemma \ref{sets} (i), $[T] = (B, T_a)$. Let $S'' = S' \cup T$. Then
$S'' \cap \Delta_C \in \Lambda_a * \Lambda_a$ therefore by lemma
\ref{sets} (iv), $S''$ meets a normal triangle with coordinates $(C, T_a)$ in
$\Delta_C$. So we have $S = (D, T_a) \cup (B, T_a) \cup (C, T_a)$,
therefore by Lemma \ref{disks}, $S$ is a $\tau$-normal triangle
(linking vertex $a$).

If $S' \cap \Delta_B$ meets a $\tau'$-normal quadrilateral $Q$, then
by lemma \ref{sets} (iii), $[Q] = (B, Q_{ad})$. We have $( S' \cup
Q ) \cap \Delta_C \in \Lambda_a * \Lambda_{ad}$ therefore by
lemma \ref{sets} (v), $S' \cup Q$ meets $\Delta_C$ in a
quadrilateral $Q'$ with $[Q'] = (C, Q_{ad})$. Let $S'' = S \cup Q \cup
Q'$. Then $S'' \cap \Delta_A \in \Lambda_d * \Lambda_d$. So that
by lemma \ref{sets} (iv), $S''$ meets $\Delta_A$ in a normal
triangle with coordinates $(A, T_d)$. Therefore $S = (D, T_a) \cup (B, Q_{ad})
\cup (C, Q_{ad}) \cup (A, T_d)$ and by Lemma \ref{disks}, $S$
is a $\tau$-normal quadrilateral (linking edge $ad$).
\end{claim}

\begin{claim}
$S$ is not a union of $\tau'$-normal quadrilaterals.\\

Without loss of generality we assume $S \cap \Delta_D \neq \phi$.
Let $Q$ be a normal quadrilateral in $S \cap \Delta_D$. Then as
the normal arcs in $\del Q$ link distinct vertices, there exists
an arc $\lambda \subset \del Q$ that belongs to $\Lambda_e$.
Assume, without loss of generality, that $Q \cap \Delta_A = \lambda \in
\Lambda_e$. By lemma \ref{sets} (iii), $Q$ meets $\Delta_A$ in a
normal quadrilateral $Q_A$ with $[Q_A] = (A, Q_{de})$. As $Q_A \cap \Delta_B
\in \Lambda_{de}$, by lemma \ref{sets} (ii), $Q_A$ meets $\Delta_B$ in
quadrilateral $Q_B$ with $Q_B = (B, Q_{de})$. Let $S' = Q_A \cup Q_B$.
Then $S' \cap \Delta_D \in \Lambda_e * \Lambda_e$, so by lemma
\ref{sets} (iv), $S'$ must meet $\Delta_D$ in a triangle with coordinates $(D, T_e)$
contradicting our assumption that $S$ is composed solely of
$\tau'$-normal quadrilaterals.
\end{claim}
\end{proof}

\section{The prism triangulation of N(F)}
Let $F$ be a closed oriented connected surface lying in an oriented 3-manifold $M$. Denote by $I$ the closed interval $[-1,1]$. Let $N(F) \cong F \times I$ be a regular neighbourhood of $F$. In this section we define a triangulation of $N(F)$, which is such that any closed connected normal surface lying in $N(F)$ is normally isotopic to $F \times \{0\}$.\\

Take a triangular disc $T$ with oriented edges. Assume the edges
are not cyclically oriented. Label the vertices $\{ v_0, v_1, v_2 \}$ of $T$ in such a way that the
edges are oriented as $\{ v_0 v_1, v_1 v_2, v_0 v_2 \}$.
In $T \times I$, let $T \times \{ -1 \}$ be identified with $[v_0,
v_1, v_2]$ labeled as above and $T \times \{ 1 \} = [w_0, w_1,
w_2]$, where $v_i$ and $w_i$ have the same image under the
projection $T \times I \rightarrow T$. Then we get a triangulation of $T \times I$, using the tetrahedra $\Delta_0 = [v_0, w_0, w_1, w_2]$, $\Delta_1 = [v_0, v_1, w_1, w_2]$ and $\Delta_2 = [v_0, v_1, v_2, w_2]$. We call this the prism triangulation of $T \times I$. (See proof of Theorem 2.10 \cite{Ha} for details.)

\begin{lemma}\label{commonprism}
Let $T_1$ and $T_2$ be triangles with non-cyclic oriented edges that
intersect in an edge $e = T_1 \cap T_2$. Assume the orientation
on the edge $e$ coming from $T_1$ is the same as that coming from
$T_2$. Let $\tau_1$ and $\tau_2$ be the prism triangulations of
$T_1 \times I$ and $T_2 \times I$ respectively. Then $\tau =
\tau_1 \cup \tau_2$ is a triangulation of $(T_1 \cup T_2)
\times I$.
\end{lemma}
\begin{proof}
In the prism triangulation of T, we note that the 1-skeleton lies in $\del T \times I$
and is the union of $\del T \times \del I$ and the edges
$\{ v_0 w_0, v_1 w_1, v_2 w_2, v_0 w_1, v_0 w_2, v_1 w_2 \}$.
Recall that the edges of $T$ were oriented as $\{ v_0 v_1, v_1
v_2, v_0 v_2 \}$. So given an oriented edge $e = [-1,1]$ of $T$,
with $e$ oriented in the direction from -1 to 1, $e \times I$ is
the union of two triangles given by the join of $e \times \{ -1\}$
with the point $(1,1)$ and the join of $e \times \{1\}$ with the point $(-1,-1)$. In particular, the triangles divide the square $e \times I$ along the diagonal from $(-1,-1)$ to $(1,1)$.

Therefore if two triangles $T_1$ and $T_2$ with oriented edges
intersect in an edge $e = T_1 \cap T_2$, where the orientation of
$e$ coming from $T_1$ is the same as that from $T_2$, then the prism
triangulation of $T_1 \times I$ and $T_2 \times I$ agree on the
intersection $e \times I$. So by taking the union $\tau_1 \cup \tau_2$ we get a
triangulation on $(T_1 \cup T_2) \times I$.
\end{proof}

We can now define the prism triangulation on $F \times I$. Firstly, we claim that given a triangulation $\tau$ of $F$ there exists a refinement $\tau'$ of $\tau$ and an orientation of the edges of $\tau'$ such that no triangle has edges oriented cyclically.

Give any orientation to the edges of the 1-skeleton of $\tau$. Let $N_\tau$ be the number of
triangles of $\tau$ with edges oriented cyclically. If $N_\tau >
0$, then take a triangle $T = [a, b, c]$ in $\tau$ with
cyclically oriented edges $\{ ab, bc, ca \}$. Let $d$ be a point
in the interior of $T$. Define the triangulation $\tau'$ as a
refinement of $\tau$ given by subdividing $T$ into the triangles
$[a, b, d]$, $[b, c, d]$, $[c, a, d]$. Orient the newly
introduced edges of the 1-skeleton as $da$, $db$ and $dc$. Then
none of the triangles in the subdivision of $T$ has cyclically
oriented edges. Therefore the number of triangles with cyclically
oriented edges in $\tau'$, $N_{\tau'} = N_\tau - 1$. So after
$N_\tau$ such refinements we obtain a triangulation of $F$ with no triangles having edges oriented cyclically.

Now by lemma \ref{commonprism}, we can patch up the prism triangulations of triangles of $F$ to get a triangulation of $F \times I$. We
call this the prism triangulation of $F \times I$, relative to the triangulation $\tau'$ of $F$.\\

We now show that any properly embedded normal surface in the prism triangulation of $T \times I$ with boundary in $\del T \times I$ is normally isotopic to $T \times \{ 0 \}$.
\begin{lemma} \label{normalinprism}
Let $T = [v_0, v_1, v_2]$ be a triangle with non-cyclic edges. Let $\tau$ be the prism
triangulation on $T \times I$. Let $S$ be a properly embedded
normal surface with $\del S \subset \del T
\times I$. Then $S$ is normally isotopic to $T \times
\{ 0 \}$.
\end{lemma}
\begin{proof}
In the prism $T \times I$, let $T \times \{-1\} = [v_0, v_1, v_2]$
and $T \times \{ 1 \} = [w_0, w_1, w_2]$ with $v_i$ and $w_i$
projecting to the same point on $T$. Then, the prism
triangulation of $T \times I$ is composed of the tetrahedra
$\Delta_0 = [ v_0, w_0, w_1, w_2 ]$, $\Delta_1 = [ v_0, v_1,
w_1, w_2 ]$ and $\Delta_2 = [ v_0, v_1, v_2, w_2 ]$. \\

Observe that $\Delta_0$ contains the face $[w_0 w_1 w_2] = T \times \{ 1 \}$, while $\Delta_2$ contains the face $[v_0 v_1 v_2] = T \times \{-1\}$. As $S$ does not intersect $T \times \del I$, $S \cap \Delta_0$ is parallel to $[w_0 w_1 w_2]$ and is therefore a union of triangles linking $v_0$. Similarly, $S \cap \Delta_2$ is a union of triangles linking $w_2$. The tetrahedron $\Delta_1$ has a pair of opposing edges $v_0 v_1$ and $w_1 w_2$ that lie in $T \times \del I$. Therefore $S \cap \Delta_1$ is a union of normal disks that separates these pair of edges and is therefore a union of normal quadrilaterals linking edge $v_0 v_1$.

Note that $\Delta_0 \cap \Delta_1 = [v_0 w_1 w_2]$ and $\Delta_1 \cap \Delta_2 = [v_0 v_1 w_2]$. So by the matching equations, the number of triangles in $\Delta_0 \cap S$ equals the number of quadrilaterals in $\Delta_1 \cap S$ which is the same as the number of triangles in $\Delta_2 \cap S$.

Let $T_0$ be a triangle in $\Delta_0 \cap S$, $Q_1$ a quadrilateral in $\Delta_1 \cap S$ and $T_2$ a triangle in $\Delta_2 \cap S$ such that $T_0$ meets $Q_1$ in $\Delta_0 \cap \Delta_1$ and $Q_1$ meets $T_2$ in $\Delta_1 \cap \Delta_2$. Then $T_0 \cup Q_1 \cup T_2 = S'$ is a connected properly embedded normal surface in $T \times I$ that projects homeomorphically onto $T$ and so $S'$ is normally isotopic to $T \times \{0\}$. Any normal surface in $T \times I$ that does not intersect $T \times \del I$ is therefore, a disjoint union of discs parallel to $T \times \{0\}$. As $S$ is a normal connected surface, $S = S'$ as required.
\end{proof}

\begin{theorem} \label{insidenbhd}
Let $F$ be a closed oriented connected surface. Let $\tau$ be a prism
triangulation of $N(F) \simeq F \times I$. Then any normal closed
connected surface $F' \subset N(F)$ is normally isotopic to $F
\times \{ 0 \}$.
\end{theorem}
\begin{proof}
Let $\tau'$ be a triangulation of $F$, and let $\tau$ be the prism triangulation of $N(F)$ relative to $\tau'$. Let $T$ be a
triangle in $\tau'$ then $\tau|_T$ is the prism triangulation on
$T \times I$. As $F'$ is normal in $\tau$ and is closed, $F' \cap (T
\times I)$ is a $\tau|_T$-normal properly embedded surface $S$
with $\del S \subset \del T \times I$ so by lemma \ref{normalinprism}, $S$ is normally isotopic to $T \times
\{ 0 \}$. As this is true for every triangle $T$ in the
triangulation $\tau'$ of $F$ and the surface $F'$ is connected,
$F'$ is normally isotopic to $F \times \{ 0 \}$.
\end{proof}

\section{Proof of Theorem \ref{incomprthm}}
As before, let $F$ be a closed oriented surface in a compact oriented 3-manifold $M$. Let $\tau_1$ be a triangulation of $N(F)$. In this section we firstly show that given any integer $W$, there exists an extension of $\tau_1$ to a triangulation $\tau$ of $M$ such that any $\tau$-normal surface $F'$ that does not lie in $N(F)$ has PL-area more than $W$. This is shown in lemma \ref{outsidenbhd}, using which we prove Theorem \ref{incomprthm}.

\begin{definition}
Let $\Gamma$ be a simplicial complex of dimension $n$. Then $|\Gamma|$ denotes the number of $n$-cells in $\Gamma$.
\end{definition}

\begin{definition}
Let $\tau$ be a triangulation of $M$. Let the $i$-weight of $F$ be defined as $w^{(i)}(F) = | F \cap \tau^{(i)}|$, where $\tau^{(i)}$ is the $i$-the skeleton of the triangulation $\tau$. Then, the PL-area of $F$ is given by the ordered pair $w(F) = (w^{(1)}(F), w^{(2)}(F))$.
\end{definition}

\begin{lemma}\label{weights}
Let $\phi$ be the refinement function (Definition \ref{refinement}) that gives the refinement of a tetrahedron $\Delta$ into 4 tetrahedra. Let $\tau$ be a triangulation of $\Delta$ consisting of the single tetrahedron $\Delta$. Let $\tau^n = \phi^n(\tau)$ be a triangulation of $\Delta$ obtained by taking $n$ iterates of $\phi$. Let $D$ be a $\tau$-normal disk. Then the 1-weight of $D$ in $\tau^n$ is greater than $n$.
\end{lemma}

\begin{proof}
As $D$ is $\tau$-normal, by Theorem \ref{mainthm}, $D$ is $\tau^n$ - normal. Let $d_n$ be the number of $\tau^n$-normal disks in $D$. Let $w_n(D) = w^{(1)} (D)$ in $\tau^n$, be the 1-weight of $D$ in $\tau^n$. As $D$
is a $\tau$-normal disk, its weight in $\tau^0 = \tau$ is greater
than equal to 3, therefore $d_0=1$ and $w_0(D) \geq 3$. Now we
claim that for $n > 0$, $d_n \geq 3 d_{n-1}$ and $w_n(D) \geq
w_{n-1}(D) + d_{n-1}$.

By lemma \ref{disks}, $D$ is divided into at least 3
$\tau^1$-normal disks on taking the refinement along $\phi$, and
its weight is increased by at least one. Therefore $d_1 \geq 3
d_0$ and $w_1 \geq w_0 + d_0$.

Similarly now, if $D$ is the union of $d_{n-1}$
$\tau^{n-1}$-normal disks then each such disk is divided into at
least 3 $\tau^n$-normal disks by taking the refinement along
$\phi$, while its weight is incremented by at least one for each
of the $\tau^{n-1}$-normal disks. Therefore $d_n \geq 3
d_{n-1}$, while $w_n \geq w_{n-1} + d_{n-1}$. So by induction, $w_n \geq w_0 + \Sigma_{i=0}^{n-1} d_i \geq w_0 + (\Sigma_{i=0}^{n-1} 3^i) d_0$\\

Therefore the weight $w_n(D) \geq 3 + 1 + 3 + 3^2 + 3^3... +
3^{n-1} > n$ for all $n > 0$.
\end{proof}

\begin{lemma} \label{outsidenbhd}
Let $\hat{F}$ be a closed surface in $M$ and let $W$ be a positive
integer. Let $\tau_1$ be a triangulation of a regular
neighbourhood $N(\hat{F})$ of $\hat{F}$ in $M$. Then, there exists an extension of
$\tau_1$ to a triangulation $\tau$ of $M$ such that for any $\tau$-normal
surface $S$ that is not contained in $N(\hat{F})$, $w^{(1)} (S) > W$.
\end{lemma}

\begin{proof}
We extend the triangulation of $\del N(\hat{F})$ given by $\tau_1$, to
a triangulation $\tau_2$ of $M - int(N(\hat{F}))$. Then, $\tau' = \tau_1 \cup \tau_2$ is a triangulation of $M$. Let $f$ be the scaling function that takes the value $W$ on tetrahedra of $\tau_2$ and the value 0 on tetrahedra of $\tau_1$. Let $\tau = \phi_f(\tau')$ be the corresponding refined triangulation. We claim that $\tau$ is the required triangulation.

Let $S$ be a $\tau$-normal surface in $M$ that is not contained
in $N(S)$. By Theorem \ref{mainthm} then, $S$ is $\tau'$-normal as well. As $S$ is not contained in $N(\hat{F})$ there exists a $\tau_2$-normal disk $D$ in $S - int(N(\hat{F}))$. By lemma \ref{weights} now, the 1-weight of $D$ in $\tau$ is greater than $W$, therefore $w^{(1)} (S) > W$ in $\tau$.
\end{proof}

We are now in a position to prove Theorem \ref{incomprthm}.
\begin{proof}
Assume $F$ is incompressible. Let $\tau$ be
any triangulation of $M$. Let $F'$ be a surface isotopic to
$F$ of minimal PL-area in the isotopy class of $F$.  If $F'$ is not $\tau$-normal then it is known that there
exists a weight minimising isotopy of $F'$, which is a
contradiction. So every minimal PL-area surface in the
isotopy class of $F$ is normal.\\

Conversely, suppose $F$ is compressible. Let $\hat{F}$ be the surface
obtained by compressing $F$ along a compressing disk.
The surface $F$ is obtained from $\hat{F}$ by
attaching a 1-handle $\gamma$.

Let $\tau'$ be a prism triangulation of $N(\hat{F}) \cong \hat{F} \times I$. Let the
1-weight of the normal surface $\hat{F} \times \{ 0 \}$ be denoted by $W$. By applying lemma
\ref{outsidenbhd}, we obtain an
extension of $\tau'$ to a triangulation $\tau$ of $M$ such that
any normal surface that does not lie in $N(\hat{F})$ has 1-weight
greater than $W$.

We can assume the 1-handle $\gamma$ is disjoint from the 1-skeleton of
$\tau$. As $F$ is obtained from $\hat{F} \times \{0 \}$ by attaching this 1-handle, the 1-weight $W = w^{(1)}(\hat{F} \times \{ 0 \}) = w^{(1)}(F)$.

Assume there exists a normal minimal surface $F'$ isotopic to $F$. By construction of $\tau$, any normal surface that does not lie in $N(\hat{F})$ has 1-weight more than $W = w(F)$. So, $F'$ lies in $N(\hat{F})$.

By Theorem \ref{insidenbhd} then, $F'$ is isotopic to a
connected component of $\hat{F}$. As $F'$ is isotopic to $F$, we have $F$ isotopic to
a connected component of $\hat{F}$. This is a contradiction as
$F$ is compressible and hence every component of $\hat{F}$ has
genus strictly lower than the genus of $F$.
\end{proof}

\acknowledgements{The author would like to thank Siddhartha Gadgil for
useful discussions and advice. The CSIR-SPM Fellowship is
acknowledged for financial support.}

\bibliographystyle{amsplain}

\end{document}